\documentclass[12pt,a4paper]{amsart}
\usepackage{PlantillaEng}
\title{Cyclic subgroups of Belk-Hyde-Matucci group $V\!\Ac$}
\author{Jos\'e Burillo}
\address{Departament de Matem\`atiques,
Universitat Polit\`ecnica de Catalunya, Jordi Girona 1-3, 08034
Barcelona, Spain} \email{pep.burillo@upc.edu}
\thanks{The first author thanks the Spanish Ministry MICINN
through grant PID2021-126851NB-I00P for their support.}
\author{Marc Felipe}
\address{}
\email{marc.felipe.alsina@gmail.com}

\newcommand\Sing{\operatorname{Sing}}
\date{}
\begin{document}
    \maketitle
\begin{abstract}
    In this paper it is proved that the Belk-Hyde-Matucci group $V\!\Ac$, a group containing
every countable abelian group, does not contain subgroups with
distorted cyclic subgroups.
\end{abstract}

\section*{Introduction}
The family of Thompson-style groups contains many examples of
finitely presented simple groups. This fact makes the Thompson
family a natural place to look for possible answers to the
Boone--Higman conjecture, which states that any finitely generated
group has a solvable word problem if and only if it embeds into a
finitely presented \emph{simple} group. See the survey
\cite{boonehigmansurvey} (and references therein) for all the
details concerning the conjecture.

In this survey \cite{boonehigmansurvey}, the authors present the
current progress on the conjecture. The authors show in this paper
that there cannot be a universal container, i.e., a finitely
presented simple group that contains every finitely presented group
with a solvable word problem. However, they analyze some particular
examples of groups, all related to the Thompson family, which show
the ability of containing large classes of groups, while at the same
time having some obstructions to contain some other groups.

One of the examples they mention is Thompson's group $V,$ which
contains many groups, such as all finite groups, (countable) free
abelian groups, and also countable nonabelian free groups). At the
same time, there are many groups which do not embed in $V,$ due to
several obstructions. One of these obstructions is the fact that
cyclic subgroups are undistorted in $V,$ so any subgroup of $V$ must
also have undistorted cyclic subgroups.

Another group mentioned in \cite{boonehigmansurvey} is
Belk-Hyde-Matucci group $V\!\Ac$, which is finitely presented and
simple, and it is shown by the authors in \cite{belkhydematucci} to
contain all countable abelian groups (such as $\mathbb{Q}$). The
purpose of this paper is to show that, as it happens in $V$, all
cyclic subgroups of $V\!\Ac$ are undistorted, and hence if we want a
group to embed in $V\!\Ac$ it had better have also undistorted
cyclic subgroups. This excludes as subgroups of $V\!\Ac$ groups such
as the Heisenberg groups (in general any finitely generated
nilpotent groups which are not virtually abelian) or
Baumslag-Solitar groups.

Moreover, we believe that the result on nondistorted cyclic
subgroups of $V\!\Ac$ has interest on its own, for instance for its
relation to the Kaplansky conjectures. On the other hand, some
distorted subgroups of $V\!\Ac$ are known, see \cite{auto}. The
group named $C$ in \cite{auto} is a subgroup of $\Ac$ and hence is
also in $V\!\Ac$, while it has a quadratically distorted copy of
$F$. Hence, $V\!\Ac$ contains a copy of $F$ which is at least
quadratically distorted.

\section{Definitions and Background}
We recall some definitions:

    A finitely generated subgroup $H$ of a finitely generated group $G$ is said to be undistorted in $G$ if the inclusion from $H$ to $G$ is a quasi-isometric embedding. That means that if we fix finite generating sets for $H$ and $G$ and we compare the minimum lengths $\ell_H(h)$ and $\ell_G(h)$ of the words on the generators (and their inverses) of $H$ and $G$ needed to write an element $h\in H$, we have the following inequality for some constant $C$ independent of $h$: $\ell_H(h)\leq C\ell_G(h)+C$.

    It is worth noting that the opposite condition $\ell_G(h)\leq C\ell_H(h)+C$ always holds when $H$ is a subgroup of $G$: by writing each generator $s_H$ of $H$ as words on the generating set for $G$, we can rewrite the word for $h$ in at most $\max\limits_{s_H}\left\{\ell_G(s_H)\right\}\cdot\ell_H(h)$ symbols.

    Let's state some well-known results about distortion:
    \begin{rmk} Let $K\leq H\leq G$ be an inclusion of finitely generated groups, and let $K$ be undistorted in $H$ and $H$ undistorted in $G$. Then $K$ is undistorted in $G$.\end{rmk}
    \begin{rmk} Let $K\leq H\leq G$ be an inclusion of finitely generated groups, and let $K$ be undistorted in $G$. Then $K$ is undistorted in $H$.\end{rmk}
    \begin{rmk} Let $x$ be an element of infinite order of the group $G$ and let $m$ be a non-zero integer. If $\<x^m\>$ is undistorted in $G$, so is $\<x\>$.\end{rmk}
    \begin{rmk} Let $a,x$ be elements of a group $G$ with $x$ having infinite order, then $\<x\>$ is an undistorted subgroup of $G$ if and only if $\<a^{-1}xa\>$ also is.\end{rmk}

    The first two remarks come from nesting two inequalities of the form $\ell_{G_1}(g)\leq C\ell_{G_2}(g)+C$. The third one from the fact that $\<x^m\>$ is undistorted in $\<x\>$, and the fourth one from the inequality $\ell_G(x^n)-2\ell_G(a)\leq\ell_G((a^{-1}xa)^n)\leq\ell_G(x^n)+2\ell_G(a)$.\\

    Now, we will focus on defining the Belk-Hyde-Matucci group $V\!\Ac$, and for that we will need $V$ and $\Ac$.

    Thompson's group $V$ is a well-known group of maps from the Cantor set $\Cfk=\{0,1\}^\NN$ to itself. Since the Cantor set is difficult to visualize, we usually identify it with the interval $[0,1]$ using the quotient map $\phi:\Cfk\to[0,1]$ which sends the sequence $x_1,x_2,x_3,\dots\in\Cfk$ to the real number $0.x_1x_2x_3\dots$ in binary. Under this identification, $V$ is the set of maps from $[0,1]$ to itself which are bijective, piecewise linear with finitely many pieces, and with the property that each piece starts and ends at a point with dyadic coordinates, and its slope is an integer power of $2$.

    The map $\phi$ is surjective, but not injective, as the dyadic numbers in $(0,1)$ have two representations as binary digits: $\frac12=0.10000000\ldots=0.01111111\ldots$. So, there is an ambiguity when defining the image of the dyadic numbers where a discontinuity occurs. Usually, this ambiguity is circumvented either by imposing right-continuity or by working with $\Cfk$ instead of $[0,1]$.

    Here, we will work with $\Cfk$, because we are interested in the behaviour in the boundaries. However, for legibility purposes, we will use numbers in $[0,1]$ to label the points of $\Cfk$, using a superscript $+$ or $-$ for dyadic numbers. So, the sequence $0,0,1,0,0,0,0,0,\dots$ is denoted by $\left(\frac18\right)^+$, the sequence $0,1,0,1,0,1,0,1,\dots$ by $\frac13$, and the sequence $0,1,0,1,1,1,1,1,\dots$ by $\left(\frac38\right)^-$. If $p\in[0,1]$, a neighborhood of $p$ is $(p-\eps,p+\eps)$, a neighborhood of $p^+$ is $[p,p+\eps)$ and a neighborhood of $p^-$ is $(p-\eps,p]$. We will often use the adjective \emph{one-sided} to refer to neighborhoods of the last two kinds.

    We present an example of an element $f$ of $V$ in \Cref{exV}. Notice that in this example, $f\left(\left(\frac12\right)^-\right)=1^-$, $f\left(\left(\frac12\right)^+\right)=\left(\frac58\right)^+$, and that $f\left(\left(\frac18\right)^-\right)=\left(\frac38\right)^-$, $f\left(\left(\frac18\right)^+\right)=\left(\frac38\right)^+$.\\
    \begin{figure}[H]
      \centering
      \includegraphicsifexists[width=0.45\linewidth]{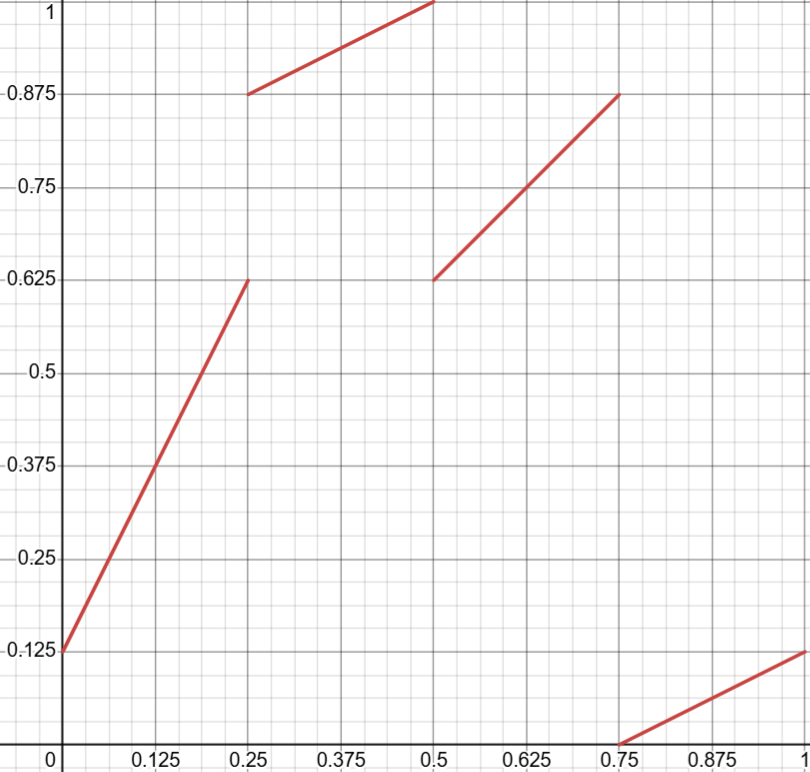}
      \caption{An element of $V$.}\label{exV}
    \end{figure}

    Brin's group $\Ac$, sometimes known as $\Aut^+(F)$, is a certain group of functions $f$ from $\RR$ to $\RR$ introduced by Brin in \cite{Brin}. By conjugating by the map $\psi:(0,1)\to\RR$ that sends linearly $\left[\frac1{2^{n+1}},\frac1{2^n}\right]$ to $[-n,-n+1]$ and $\left[1-\frac1{2^n},1-\frac1{2^{n+1}}\right]$ to $[n-1,n]$ for each integer $n\geq1$, we obtain an isomorphic group of functions from $[0,1]$ to $[0,1]$,
    whose elements are the functions $f$ that satisfy the following conditions:
    \begin{itemize}
        \item[$-$] $f(0)=0$ and $f(1)=1$.
        \item[$-$] $f$ is continuous and piecewise linear, but breakpoints may accumulate only at $0$ and $1$.
        \item[$-$] Each piece has a power of 2 as their slope, and starts and ends at points with dyadic coordinates.
        \item[$-$] In a neighborhood of $0$, $f$ satisfies $f(2x)=2f(x)$, and in a neighborhood of $1$, $f$ satisfies $f(2x-1)=2f(x)-1$.
    \end{itemize}

    Visually, the last condition states that the graph of the function experiences a fractal-like behavior in a neighborhood of the corresponding point, where it remains invariant under a $\times2$ zoom near $(0,0)$ and near $(1,1)$. This behavior near 0 and 1 is due to the behavior in the neighborhoods of $\infty$ and $-\infty$ in the real-line interpretation of $\Aut F$ used by Brin in \cite{Brin},  and corresponds to the fact that maps are \emph{periodic} there,  i.e., satisfy $f(t+1)=f(t)+1$.

    Now that the domain and codomain are $[0,1]$, it is easy to make them be the Cantor set: for each dyadic $p\in[0,1]$, just let $f(p^-)$ be $(f(p))^-$ and $f(p^+)$ be $(f(p))^+$.

    We present an example of an element of $\Ac$:\\
    \begin{figure}[H]
      \centering
      \includegraphicsifexists[width=0.45\linewidth]{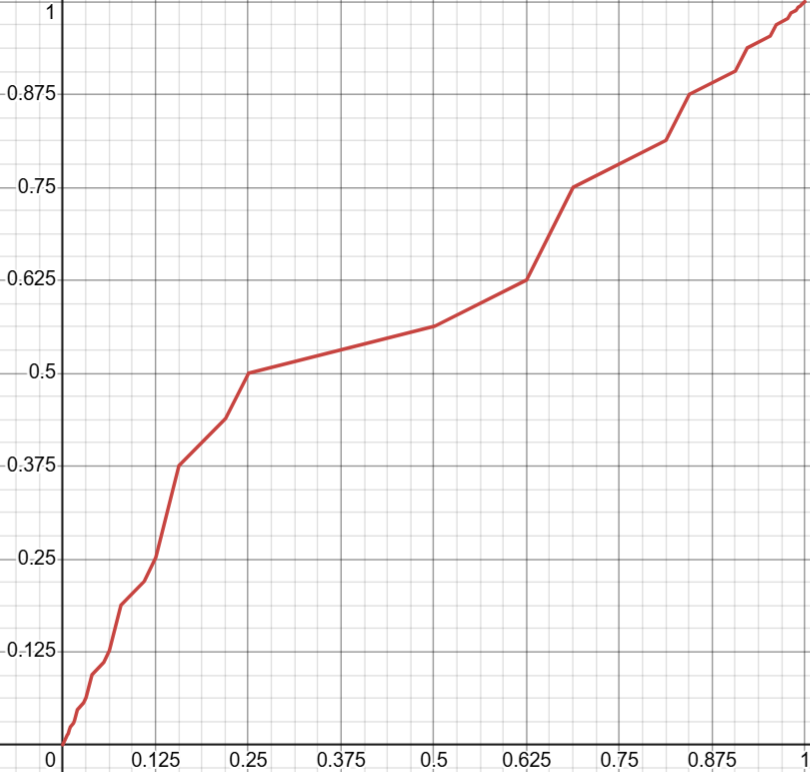}
      \caption{An element of $\Ac$.}
    \end{figure}

    It is worth noting that, even though an element of $\Ac$ may have infinitely many linear pieces, there is only a finite number of slopes these pieces can have, as the ones near $0$ and $1$ are bound to have the same slope as a further segment.\\

    The Belk-Hyde-Matucci group $V\!\Ac$, introduced in \cite{belkhydematucci}, is the group of functions that is generated by $V$ and $\Ac$, when both are seen as functions from $\Cfk$ to $\Cfk$. They proved that this group is finitely generated, as $V$ and $\Ac$ are, and is the group of functions satisfying the following:

    \begin{itemize}
        \item[$-$] $f$ is bijective and piecewise linear, but breakpoints may accumulate at a finite number of points called singularities.
        \item[$-$] Each piece has a power of 2 as their slope, and starts and ends at points with dyadic coordinates.
        \item[$-$] Each singularity is of the form $p^+$ or $p^-$ for some dyadic number $p$. The image of the singularity is also of this form, with the same superscript.
        \item[$-$] If $p^\pm$ is a singularity and $f(p^\pm)=q^\pm$, then in a sufficiently small (one-sided) neighborhood of $p^\pm$, the function is continuous when drawn as a function from $[0,1]$ to itself and satisfies $f\circ L_p=L_q\circ f$ where $L_r(x)=2(x-r)+r$.
    \end{itemize}

    We will say that a neighborhood of a singularity is regular if it small enough to satisfy the last condition.

    \begin{figure}[H]
        \centering
        \includegraphicsifexists[width=0.45\linewidth]{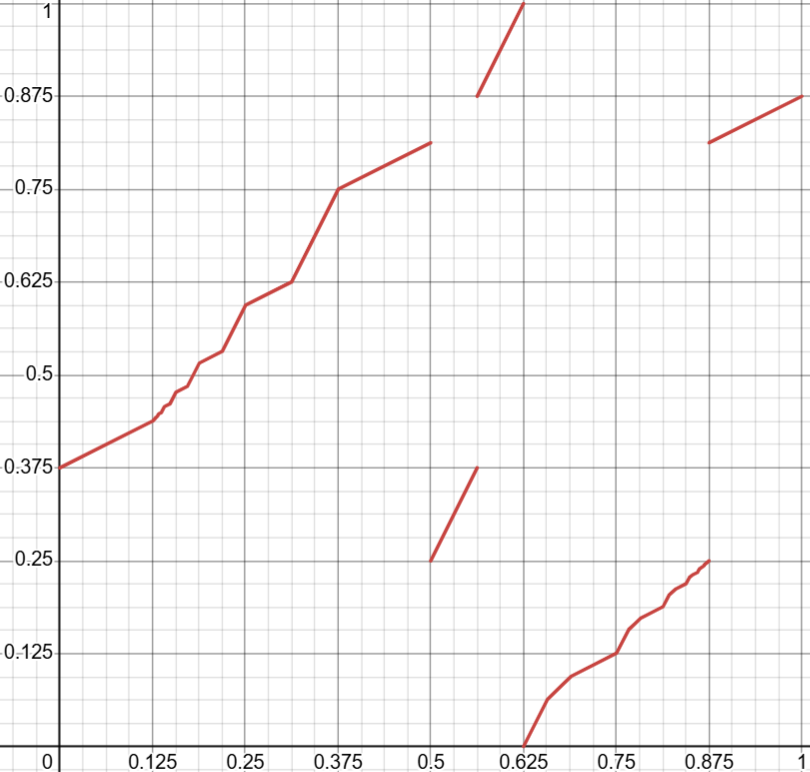}
        \caption{Element of $V\!\Ac$ with two singularities, one at $\left(\frac18\right)^+$ and one at $\left(\frac78\right)^-$.}
    \end{figure}

    When multiplying two elements of $V\!\Ac$, we can approximately locate its singularities. Indeed, if $f,g\in V\!\Ac$ and $f\cdot g\,(=g\circ f)$ has a singularity at $s\in\Cfk$, then it can only be due to the following two reasons: $f$ itself has a singularity at $s$, or $f$ sends $s$ to $t$, where $t$ is a singularity of $g$. That is, $s\in\Sing(f)\cup f^{-1}(\Sing(g))$. On the other hand, if only one of those two things happen, that is $s\in\Sing(f)\vartriangle f^{-1}(\Sing(g))$, $s$ will be a singularity of $f\cdot g$, but may not be if both happen at the same time, as the singularity created by $f$ may be undone by the singularity of $g$ at $f(s)$. So, we can conclude the following:
    \begin{rmk} $f^{-1}(\Sing(g))\vartriangle\Sing(f)\subseteq\Sing(fg)\subseteq f^{-1}(\Sing(g))\cup\Sing(f)$, where $\vartriangle$ denotes the symmetric difference of sets.\end{rmk}

    A direct consequence of this is $$|\Sing(fg)|\leq|\Sing(f)|+|f^{-1}(\Sing(g))|=|\Sing(f)|+|\Sing(g)|$$

\section{Infinite cyclic subgroups are undistorted}
    We will now prove the main theorem of this paper:

    \begin{thm} Let $f\in V\!\Ac$ have infinite order. Then $\<f\>$ is undistorted in $V\!\Ac$.
    \end{thm}

    Since $\ell_{\<f\>}(f^k)=|k|$, we need to prove that $|k|\leq C\ell_{V\!\Ac}(f^k)+C$, which is the same as seeing that $\ell_{V\!\Ac}(f^k)$ is at least linear in $|k|$, and that is what we will do. In fact, since $\ell_G(f)=\ell_G(f^{-1})$, it is enough to prove it for positive $k$.\\

    Now, we consider the orbits of the singularities of an element $f\in V\!\Ac$. That is, the sequence $\dots,f^{-2}(s),f^{-1}(s),s,f(s),f^2(s),\dots$ for every singularity $s$.
    \begin{lmm} There exists a conjugate of $f$ with at most one singularity per orbit.
    \end{lmm}
    \begin{prf}
        If the orbit of a singularity is infinite, label it $\dots,s_{-1},s_0,s_1,s_2,\dots,s_{m-1},s_m,s_{m+1},\dots$, with $f(s_i)=s_{i+1}$ for every integer $i$. If the orbit is periodic, label it $s_0,s_1,\dots,s_m,\dots,s_k$ with $f(s_i)=s_{i+1}$ and $f(s_k)=s_0$. In any case, let the orbit be labelled in a way that all the finitely many singularities in this orbit lie in the subset $\{s_0,s_1,\dots,s_m\}$, with $s_0$ and $s_m$ being singularities. Note that $s_1,s_2,\dots,s_{m-1}$ may or may not be singularities of $f$.

        Take a regular (one-sided) neighborhood $U$ of $s_m$, which by definition does not contain any other singularity of $f$. The image of this neighborhood under $f$ is a neighborhood $f(U)$ of $s_{m+1}$. Let $f'$ be an element of $V\!\Ac$ that coincides with $f$ everywhere except in $U$. We choose $f'$ so that $f'(s_m)=s_{m+1}$ and there are no singularities nor discontinuities of $f'$ in $U$, which is easy to do due to transitivity in Thompson's groups. So, the singularities of $f'$ are the same as the ones in $f$, except it lacks $s_m$.

        Now, let $a=f\cdot f'^{-1}$. Since $a$ is the identity everywhere outside $U$, the only singularities that $a$ can have are in $$\left(f^{-1}(\Sing(f'^{-1}))\cup\Sing(f)\right)\cap U=\left(f^{-1}(\Sing(f'^{-1}))\cap U\right)\cup\left(\Sing(f)\cap U\right)$$ Note that for any function $g$, we have $\Sing(g^{-1})=g(\Sing(g))$, which we can see intuitively or via $$g(\Sing(g))\vartriangle\Sing(g^{-1})\subseteq\Sing(g^{-1}\cdot g)=\Sing(\Id)=\void$$ So, $f^{-1}(\Sing(f'^{-1}))=f^{-1}(f'(\Sing(f')))=a^{-1}(\Sing(f'))$. As $\Sing(f')$ does not intersect with $U$ and $a$ is the identity outside of $U$, we have that $f^{-1}(\Sing(f'^{-1}))\cap U=\void$, so the only possible singularity is in $\Sing(f)\cap U=\{s_m\}$. It is indeed a singularity as it belongs to $\Sing(f)$ but not to $f^{-1}(\Sing(f'^{-1}))$. Since $f(s_m)=s_{m+1}$ and $f'^{-1}(s_{m+1})=s_m$, we see that $s_m$ is a fixed point of $a=f\cdot f'^{-1}$.

    Knowing that, we study the singularities of $a^{-1}fa$, which is equal to $f'a$. We have that $f'^{-1}(\Sing(a))=\{s_{m-1}\}$ and $\Sing(f')=\Sing(f)\trec\{s_m\}$, so $a^{-1}fa$ has the same singularities as $f$, except it does not have $s_m$ anymore and it may either gain, retain or lose $s_{m-1}$ as a singularity. In any case, the singularities in other orbits have been unaffected and the interval of possible locations for singularities in the orbit in question has shrunk by at least one unit, at the cost of changing $f$ for a conjugate. Repeating this process, eventually we either reach the case $m=0$ or a state where no singularities are left in this orbit. Doing the same for other orbits, we obtain the desired result.
    \end{prf}

    Since the distortion of the subgroups generated by $f$ and one of its conjugates is the same, we may suppose that $f$ has at most one singularity per orbit.\\

    The elements of $V\!\Ac$ without singularities are precisely the elements of $V$. For those, we have the following lemma by Higman \cite{Higman}:
    \begin{lmm} For any element $f$ of infinite order in $V$, there is an integer $n$ such that for the reduced tree-pair diagram $(S,T,\pi)$ for $f^n$, there is a leaf $i$ in the source tree $S$ which is paired with a leaf $j$ in the target tree $T$ so that $j$ is a proper child of $i$.\end{lmm}
    By the lemma, when restricted to the interval representing $i$, $f^n$ will be a linear contraction, so it will have a fixed point (which may not be dyadic) with a segment, emanating from at least one side of it, with a slope of $2^{-m}$ for $m\geq1$. Since the slopes on the generators of $V$ and $\Ac$ have a lower and upper bound, we will need at least a linear-in-$|k|$ quantity of them in order to construct the corresponding segment of $(f^n)^k$ with slope $(2^{-m})^k$, so $\<f^n\>$ is undistorted and so is $\<f\>$.

    We are left with the elements with singularities. We will first do the case where there is an infinite orbit. Let $s_0$ be the only singularity in this orbit and $\dots,s_{-2},s_{-1},s_0,s_1,s_2,\dots$ be its orbit, where $s_i=f^i(s_0)$. Then we claim by induction that $\Sing(f^n)=\{s_{-n+1},\dots,s_{-1},s_0\}$. The base case $\Sing(f)=\{s_0\}$ is clear, and for the induction step, we know that
    $$f^{-1}(\Sing(f^n))\vartriangle\Sing(f)\subseteq\Sing(f^{n+1})\subseteq f^{-1}(\Sing(f^n))\cup\Sing(f)$$
    Since $\Sing(f^n)=\{s_{-n+1},\dots,s_{-1},s_0\}$ by hypothesis, applying $f^{-1}$ to it transforms it into $\{s_{-n},\dots,s_{-2},s_{-1}\}$. This set is disjoint with $\Sing(f)=\{s_0\}$, so the left-hand side and the right-hand side of the inclusions coincide. So, $\Sing(f^{n+1})=\{s_{-n},\dots,s_{-1},s_0\}$, completing the induction.

    This means that $f$ has infinite order and the number of singularities in the powers of $f$ grow at least linearly with the exponent. Since multiplying by a generator of $V$ or $\Ac$ can only add a maximum of $2$ singularities, we see that $\<f\>$ is undistorted in $V\!\Ac$.\\

    We will now study the case with a singularity $s_0$ in a finite orbit, which we label as\newline $s_{-m+1},s_{-m+2},\dots,s_{-1},s_0$, where $f(s_i)=s_{i+1}$ and $f(s_0)=s_{-m+1}$. Applying the same induction as before, we know that $\Sing(f^n)=\{s_{-n+1},\dots,s_{-1},s_0\}$ for $n\leq m$, after which the induction breaks, as $f^{-1}(\Sing(f^n))$ and $\{s_0\}$ are no longer disjoint. For $n=m$, it is particularly interesting, as we see that $f^m(s_i)=s_i$, so we have singularities that are also fixed points.

    Using the lemma stated and proved below, we see that $\<f^m\>$ is undistorted and $\<f\>$ is as well, concluding the proof that infinite cyclic subgroups of $V\!\Ac$ are undistorted.

    \begin{lmm} Let $g\in V\!\Ac$ be a function with a singularity on a fixed point. Then $g$ has infinite order and $\<g\>$ is undistorted in $V\!\Ac$.
    \end{lmm}

    \begin{prf}
        Let $s$ be this singularity and take a regular neighborhood $U$ of it. Since $g$ is continuous in this interval, there can only happen three situations: $g(x)<x$ for all $x\in U\trec\{s\}$, $g(x)>x$ for all $x\in U\trec\{s\}$ or there exists $x_0\in U\trec\{s\}$ with $g(x_0)=x_0$. We solve the last case first: consider the segment containing $x_0$. If it has slope $1$, then the next segment (in the direction towards $s$) satisfies that it touches the line $y=x$ and its slope is not $1$. Anyhow, we have found a segment of a slope different than $1$ with a fixed point (which may or may not be dyadic). Studying the slope of the corresponding segment of $g^k$, we see that $g$ has infinite order and that $\<g\>$ is undistorted.

        \begin{figure}[H]
            \centering
            \includegraphicsifexists[width=0.45\linewidth]{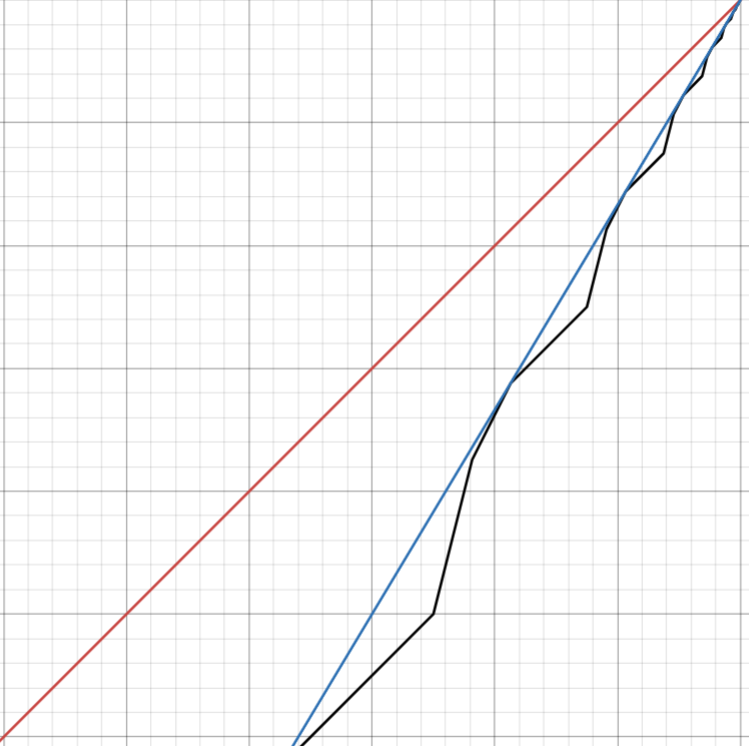}
            \caption{$y=g(x)$ in black, $y=x$ in red and $y=p-\lambda(p-x)$ in blue. The singularity $s=p^-$ is at the top right.}\label{puntfix}
        \end{figure}

        Now we do some casework on whether $s=p^-$ or $s=p^+$ and on whether $g(x)<x$ or $g(x)>x$ for all $x\in U$. In any case, due to the nature of a singularity we have that for every $x$ in $U$, the points $(p,p)$, $(x,g(x))$ and $\left(\frac{x+p}2,\frac{g(x)+p}2\right)$ are collinear. In the $s=p^-$, $g(x)<x$ case (see \Cref{puntfix}), we can take the minimum slope $\lambda$ between the points $(x,g(x))$ and $(p,p)$ for $x$ in $U$ and this minimum is achieved. The fact that $g(x)<x$ tells us that $\lambda>1$, and $\frac{p-g(x)}{p-x}\geq\lambda$ tells us that $g(x)\leq p-\lambda(p-x)$, which means that the graph of $g$ lies below the line through $(p,p)$ and slope $\lambda$.

        Since $g$ is increasing in $U$, we see that for $x$ sufficiently close to $p$ (in the sense that $g(x)$ is still in $U$), we have that $g^2(x)=g(g(x))\leq p-\lambda(p-g(x))\leq p-\lambda(p-(p-\lambda(p-x)))=p-\lambda^2(p-x)$. Repeating this, we see that there exists some neighborhood of $p^-$ where $g^k(x)\leq p-\lambda^k(p-x)$. This forces $g^k$ to have a segment with slope at least $\lambda^k$, which can only be constructed using a linear-in-$k$ number of generators, so $g$ has infinite order and $\<g\>$ is undistorted.

        The case where $g(x)>x$ and $s=p^+$ is done in the same way. The remaining cases are $g(x)<x$ with $s=p^+$ and $g(x)>x$ with $s=p^-$, which are done analogously by taking the maximum slope instead.
    \end{prf}

    This ends the proof that infinite cyclic subgroups of $V\!\Ac$ are undistorted. Using one of the remarks in the previous section, the next follows as a corollary:

    \begin{crl} $V\!\Ac$ does not contain any group with cyclic subgroups that are distorted.
    \end{crl}

    This corollary invalidates $V\!\Ac$ as a candidate to a universal group for groups with solvable word problem, as there exist groups with solvable word problem and distorted cyclics.\\

    As a bonus, we present the following result:

    \begin{prp} If $f\in V\!\Ac$ is of finite order, then it is conjugate to an element of $V$.
    \end{prp}

    \begin{prf}
        We know that we can take a conjugate of $f$ with at most one singularity per orbit. However, for both infinite and periodic orbits, we have seen that a single singularity forces $f$ to have infinite order, which is not the case. Therefore, this conjugate of $f$ cannot have singularities and belongs to $V$.
    \end{prf}

\bibliography{pepsrefs}
\bibliographystyle{plain}

\end{document}